\documentclass{amsart}
\usepackage{amssymb,xspace}
\usepackage[all]{xy}
\newtheorem{thm}{Theorem}[section]
\newtheorem{lem}[thm]{Lemma}
\newtheorem{prop}[thm]{Proposition}
\newtheorem{cor}[thm]{Corollary}

\theoremstyle{definition}
\newtheorem{rk}[thm]{Remark}
\newtheorem{defn}[thm]{Definition}

\numberwithin{equation}{section}

\newcommand{\OO}{{\mathcal O}}
\newcommand{\MC}{\lM_\C}
\newcommand{\tildeMC}{\wtl\lM_\C}
\newcommand{\C}{\mathbb{C}}
\renewcommand{\P}{\mathbb{P}}
\newcommand{\MM}{\mathfrak{M}}
\newcommand{\R}{\mathbb{R}}
\newcommand{\LL}{\mathbb{L}}
\newcommand{\Jtilde}{\widetilde{J}} 
\newcommand{\Qtilde}{\widetilde{Q}} 
\newcommand{\Atilde}{\widetilde{A}} 
\newcommand{\That}{{\mathcal{T}}} 

\newif\ifrem\remtrue
\def\mat#1{\ensuremath{#1}\xspace}
\def\DMO{\DeclareMathOperator}
\def\set#1{\mat{\{#1\}}}
\def\sets#1#2{\mat{\{#1\mid#2\}}}
\def\ang#1{\mat{\left\langle #1\right\rangle}}
\def\case#1{\begin{cases}#1\end{cases}}
\def\wtl{\widetilde}
\def\ub{\overline}

\def\cC{\mathbb{C}}
\def\cQ{\mathbb{Q}}
\def\cH{\mathbb{H}}
\def\cL{\mathbb{L}}
\def\cN{\mathbb{N}}
\def\cP{\mathbb{P}}
\def\cR{\mathbb{R}}
\def\cZ{\mathbb{Z}}
\def\lM{\mathcal{M}}
\def\lP{\mathcal P}
\def\al{\mat{\alpha}}
\def\be{\mat{\beta}}
\def\eps{\mat{\varepsilon}}
\def\De{\mat{\Delta}}

\def\La{\mat{\Lambda}}
\def\la{\mat{\lambda}}
\def\hi{\mat{\chi}}
\def\si{\mat{\sigma}}
\def\vi{\mat{\varphi}}
\def\ze{\mat{\zeta}}
\def\gM{\mathfrak{M}}

\DMO\Hom{Hom}
\DMO\End{End}
\DMO\GL{GL}
\DMO\Aut{Aut}
\DMO\Exp{Exp}
\DMO\Log{Log}
\DMO\Pow{Pow}
\DMO\Id{Id}

\def\crit{\operatorname{crit}}
\def\udim{\operatorname{\underline\dim}}
\def\vir{\mathrm{vir}}
\def\im{\mathrm{im}}
\def\re{\mathrm{re}}

\def\pser#1{[\![#1]\!]} 

\def\dd{\mat{\partial}}
\def\ms{\backslash} 
\def\sb{\subset}
\def\xx{\times}
\def\n#1{\mat{\lvert#1\rvert}}
\def\mto{\mapsto}
\def\ts{\otimes}

\def\inv{^{-1}}
\def\dual{^\vee}
\def\oh{\mat{\frac12}}
\def\ie{i.e.\ }

\def\GG{G}

\usepackage{hyperref}
\begin{document}

\remfalse
\title[Motivic Donaldson--Thomas invariants of the conifold]{Motivic Donaldson--Thomas invariants of the conifold and the refined topological vertex}

\author[Morrison]{Andrew Morrison}
\email{andrewmo@math.ubc.ca}
\author[Mozgovoy]{Sergey Mozgovoy}
\email{mozgovoy@maths.ox.ac.uk}
\author[Nagao]{Kentaro Nagao}
\email{kentaron@math.nagoya-u.ac.jp}
\author[Szendr\H oi]{Bal\'azs Szendr\H oi} 
\email{szendroi@maths.ox.ac.uk}

\begin{abstract}
We compute the motivic Donaldson--Thomas theory of the resolved conifold, in all chambers 
of the space of stability conditions of the corresponding quiver. The answer is a product formula
whose terms depend on the position of the stability vector, generalizing known 
results for the corresponding numerical invariants. Our formulae imply in particular a motivic form of the 
DT/PT correspondence for the resolved conifold. The answer for the motivic
PT series is in full agreement with the prediction of the refined topological vertex formalism. 
\end{abstract}

\maketitle

\thispagestyle{empty}

\section*{Introduction}

A {\it Donaldson-Thomas} (DT) {\it invariant} of a Calabi-Yau $3$-fold $Y$ is a counting invariant of 
coherent sheaves on $Y$, introduced in \cite{thomas_holomorphic} as a holomorphic analogue of the Casson invariant 
of a real $3$-manifold. A component of the moduli space of (say stable) coherent sheaves on $Y$ carries 
a symmetric obstruction theory and a virtual fundamental cycle \cite{behrend_intrinsic,behrend_symmetric}. 
A DT invariant of a compact $Y$ is then defined as the integral of the constant function $1$ over the virtual 
fundamental cycle of the moduli space.

It is known that the moduli space of coherent sheaves on $Y$ can be locally described as the critical locus of a 
function, the {\it holomorphic Chern--Simons functional} (see \cite{joyce_theory}). Behrend provided 
a description of DT invariants in terms of the Euler characteristic of the {\it Milnor fiber} of the CS 
functional~\cite{behrend_donaldson-thomas}. Inspired by this result, the proposal of \cite{kontsevich_stability,behrend_motivic} was  
to study the {\it motivic Milnor fiber} of the CS functional as a motivic refinement of the DT invariant. Such a refinement 
had been expected in string theory \cite{iqbal_refined, dimofte_refined}.

The purpose of this paper is to show how the ideas of Szendr\H oi~\cite{szendroi_non-commutative} and Nagao and Nakajima~\cite{nagao_counting}
can be used to study the motivic refinement of DT theory  and related enumerative theories 
associated to the local conifold $Y=\OO_{\P^1}(-1,-1)$, the threefold total space over~$\P^1$ of the rank two 
bundle $\OO_{\P^1}(-1)\oplus\OO_{\P^1}(-1)$. In~\cite{szendroi_non-commutative}, it was realized that a counting problem
closely related to the original DT counting on $Y$ can be formulated algebraically, in terms of counting representations
of a certain quiver with potential (see below), the so-called conifold quiver. It was also conjectured there that the algebraic and
geometric counting problems are related by wall crossing. The paper~\cite{nagao_counting} realized this, by
\begin{itemize}
\item describing the natural chamber structure on the space of stability parameters of the conifold quiver,
\item finding chambers which correspond to geometric DT and stable pair (PT), as well as algebraic noncommutative DT invariants, and
\item computing the generating function of Donaldson-Thomas type invariants for each chamber.
\end{itemize}
In this paper, we consider motivic refinements of these formulae. The motivic refinement is given by the motivic class of 
vanishing cycles of the conifold potential. This virtual motive ``motivates'' DT theory and its variants (PT, NCDT) in the sense 
that its Euler characteristic specialization is the corresponding enumerative invariant of the moduli space. 

The main result of this paper is the computation of the generating series of these virtual motives in all chambers of the space of 
stability conditions. We constantly use the torus action on $Y$, together with a result of~\cite{behrend_motivic}. 
We use the factorization property of \cite{kontsevich_stability,kontsevich_cohomological, mozgovoy_motivica, nagao_wall-crossing}. 
We also need one explicit evaluation, Theorem~\ref{eq univ exp form}; 
we give two proofs of that result, one using an explicit calculation of the generating series of 
motives of a certain space of matrices, another relying on a further ``dimensional reduction'' to a problem on a tame quiver.

At large volume, our result agrees (up to a subtlety involving the Hilbert scheme of points) with the refined topological vertex 
formulae of~\cite{iqbal_refined}, also discussed in~\cite{dimofte_refined} in this context. 

The motives considered here exist globally over the moduli spaces. Thus our point of view is slightly different from that 
of~\cite{kontsevich_stability}, whose general framework involves building motivic invariants from local data. The results here are fully compatible with theirs, but the proofs do not depend on the partially conjectural setup of~\cite{kontsevich_stability}, in particular their integration map.

As well as a motivic refinement, there is also a ``categorification'' given by the mixed Hodge module of vanishing cycles of the 
superpotential; compare~\cite{kontsevich_cohomological}. Our results can also be interpreted as computing the generating series 
of E-polynomials of this categorification.

\subsection*{Main result}
Let $J=J(Q,W)$ be the non-commutative crepant resolution of the conifold, a quiver algebra
with relations coming from the Klebanov--Witten potential $W$ (see Section~\ref{sec:coniquiver} for details). 
Let $\Jtilde=J(\Qtilde,W)$ be 
the framed algebra given by adding the new vertex $\infty$ to the quiver of $J$.
In \cite{nagao_counting}, the authors introduce a notion of $\ze$-(semi)stability of 
$\Jtilde$-modules $\wtl{V}$ with $\dim \wtl{V}_\infty\leq 1$ for a stability 
parameter $\ze=(\ze_0,\ze_1)\in \mathbb{R}^2$. 

Let $\al\in\cN^2$ and let ${\MM}_\ze(\Jtilde,\al)$ be the moduli space of $\ze$-stable $\Jtilde$-modules $\wtl{V}$, with $\udim\wtl{V}=(\al,1)$.
We want to compute the motivic generating series
\[
Z_\ze(y_0,y_1)=
\sum_{\al\in\cN^2}
\Bigl[{\MM}_\ze\bigl(\Jtilde,\al\bigr)\Bigr]_\vir \cdot y_0^{\al_0}y_1^{\al_1}\in \MC[[y_0, y_1]].
\]
Here $[\bullet]_\vir$ denotes the {\it virtual motive} (see Section \ref{subsec_11}), an element of a suitable ring of motives $\MC$. 

As proved in~\cite{nagao_counting}, the stability parameter space $\R^2$ is a countable union of chambers, within which 
the moduli spaces and therefore the generating series $Z_\ze$ remain unchanged. The chambers are separated
by a set of walls, defined by a set of positive roots 
\[\Delta_+=\Delta^{\re}_+\sqcup \Delta^{\im}_+,\]
where
\begin{align*}
\Delta^{\re}_+&=\{(i,i-1)\mid i\geq 1\}\cup\{(i-1,i)\mid i\ge 1\},\\
\Delta^{\im}_+&=\{(i,i)\mid i\geq 1\}.
\end{align*}
To each element $\al=(\al_0, \al_1)\in \Delta_+$, we associate a finite product as follows: for 
real roots $\al\in \Delta^{\re}_+$, put
\[ Z_\al(-y_0, y_1) = \prod_{j=0}^{\al_0-1} \left(1-\cL^{-\frac{\al_0}2+\oh+j}y_0^{\al_0}y_1^{\al_1}\right), 
\]
whereas for imaginary roots $\al\in\Delta^{\im}_+$, put
\[ Z_\al(-y_0, y_1) 
=\prod_{j=0}^{\al_0-1}
\left(1-\cL^{-\frac{\al_0}2+1+j}y_0^{\al_0}y_1^{\al_1}\right)^{-1}
\left(1-\cL^{-\frac{\al_0}2+2+j}y_0^{\al_0}y_1^{\al_1}\right)^{-1}. 
\]
Our main result is the following product formula: 

{\def\thethm{1}
\begin{thm}\label{announce_thm_main_res}
For $\ze\in\R^2$ not orthogonal to any root, 
$$Z_\ze(y_0,y_1)=\prod_{\substack{\al\in\De_+\\ \ze\cdot\al<0}}Z_{\al}(y_0, y_1).$$
\end{thm}}

By~\cite{behrend_donaldson-thomas, behrend_motivic}, the specialization $Z_\ze(y_0, y_1)|_{\LL^\oh\to 1}$ is the DT-type series at the generic 
stability parameter $\ze$, computed in special cases 
in~\cite{behrend_super-rigid,maulik_gromov-witten, szendroi_non-commutative, young_computing} 
and in general in~\cite{chuang_wall, nagao_counting}. Previously, all these results have been obtained by torus localization; we
obtain new proofs of all these formulae. Since~\cite{nagao_counting} identifies the DT and PT chambers for~$Y$, we in particular 
get motivic results for these two chambers. 

{\def\thethm{2}
\begin{cor}\label{announce_thm_main_cor}
The refined DT and PT series of the resolved conifold are given by the formulae
\[
Z_{\rm PT}(-s, T) = \prod_{m\geq 1}\prod_{j=0}^{m-1} \left(1- \cL^{-\frac m2+\oh+j} s^m T \right)
\]
and
\[Z_{\rm DT}(-s, T) = Z_{\rm PT}(-s, T)\cdot\prod_{m\geq 1}\prod_{j=0}^{m-1}\left(1-\cL^{-\frac m2+1+j}s^m\right)^{-1} \left(1-\cL^{-\frac m2+2+j}s^m\right)^{-1},
\]
written in the geometric variables $s,T$, with $s$ representing the point class and $T$ representing the curve class as usual.
\end{cor}}

Thus in particular we compute the first instance of a motivic DT partition function for the original geometric problem of rank-1 invariants of ideal sheaves of points and curves~\cite{maulik_gromov-witten} where 
a curve is present. Corollary~\ref{announce_thm_main_cor} also proves the motivic version
of the DT/PT wall crossing formula. These results are compared to the expectation from the refined topological vertex~\cite{iqbal_refined}
in Section~\ref{sec:ref-top-vx}. 

\section{Preliminaries}

\subsection{Motives}\label{subsec_11}
We are working in a version of the ring of motivic weights: let $\MC$ denote the $K$-group of the category of effective Chow motives over $\C$, extended by $\LL^{-\oh}$, where $\LL$ is the Lefschetz motive. It has a natural structure of a 
\la-ring \cite{getzler_mixed,heinloth_note} (see Section \ref{lambda} for the definition of a \la-ring) with $\si$-operations defined by $\si_n([X])=[X^n/S_n]$ and $\si_n(\cL^\oh)=\cL^{\frac n2}$. There is a dimensional completion \cite{behrend_motivica}
$$\tildeMC=\MC\pser{\cL\inv},$$
which is also a \la-ring. Note that in this latter ring, the elements $(1-\cL^n)$, and therefore the motives of general 
linear groups, are invertible. The rings $\MC\sb\tildeMC$ sit in larger rings $\MC^{\hat{\mu}}\subset\tildeMC^{\hat{\mu}}$ of equivariant motives, 
where $\hat\mu$ is the group of all roots of unity~\cite{looijenga_motivic}.

The map that sends a smooth projective variety $X$ to its $E$-polynomial
$$E(X,u,v)=\sum_{p,q\ge0}(-1)^{p+q}\dim H^{p,q}(X,\cC)u^pv^q$$
can be extended to the ring homomorphism $E:\tildeMC\to\cQ[u,v]\pser{(uv)^{-\oh}}$. This map is a \la-ring homomorphism, where the \la-ring structure on $\cQ[u,v]\pser{(uv)^{-\oh}}$ is given by Adams operations (see Section \ref{lambda})
$$\psi_n(f(u,v))=f(u^n,v^n).$$
The map $E:\MC\to\cQ[u,v,(uv)^{-\oh}]$ can be further specialized to the Euler number $e:\MC\to\cQ$ by $u\mto1,v\mto1,(uv)^{-\oh}\mto1$.

\begin{rk}
Note that the Euler number specialization of $\cL^\oh$ is $\cL^\oh\mto1$. This differs from the conventions of \cite{behrend_motivic}, where the specialization is $\cL^\oh\mto-1$. This difference results from the fact that \cite{behrend_motivic} uses the \la-ring structure on $\MC$ with $\si_n(-\cL^\oh)=(-\cL^\oh)^n$ \cite[Remark 1.7]{behrend_motivic}.
\end{rk}

Let $f\colon X \to \cC$ be a regular function on a smooth variety X.
Using arc spaces, Denef and Loeser \cite{denef_geometry,looijenga_motivic} define the motivic nearby 
cycle $[\psi_f]\in \lM_\cC^{\hat{\mu}}$ and the motivic vanishing cycle 
\[ [\varphi_f]=[\psi_f]-[f\inv(0)]\in \lM_{\cC}^{\hat{\mu}}\]
of~$f$. Note that if $f=0$, then $[\vi_0]=-[X]$. The following result was proved in \cite[Prop.~1.11]{behrend_motivic}.

\begin{thm}
\label{thr:bbs}
Let $f:X\to\cC$ be a regular function on a smooth variety X.
Assume that $X$ admits a $\cC^*$-action such that $f$ is $\cC^*$-equivariant \ie $f(tx)=tf(x)$ for $t\in\cC^*$, $x\in X$, and 
such that there exist limits $\lim_{t\to0}tx$ for all $x\in X$. Then
$$[\vi_f]=[f\inv(1)]-[f\inv(0)]\in\lM_\cC\sb\lM_\cC^{\hat\mu}.$$
\end{thm}

Following~\cite{behrend_motivic}, we define the {\it virtual motive} of $\crit(f)$ to be
\[
[\crit(f)]_\vir = -(-\cL^\oh)^{-\dim X}[\varphi_f]\in \lM^{\hat{\mu}}_\cC.
\]
Thus for a smooth variety $X$ with $f=0$, 
\[
[X]_\vir= [\crit(0_X)]_\vir=(-\cL^\oh)^{-\dim X}\cdot [X].
\]

\begin{rk} The ring $\MC$ is known to be a homomorphic image of the naive motivic 
ring $K_0({\rm Var}_\C)[\LL^{-\oh}]$. Some of the works cited above work in this ring; the quoted constructions
and results carry over to $\MC$ under this ring homomorphism. We prefer to work in $\MC$ since that is known to 
be a \la-ring. 
\end{rk}

\subsection{\texorpdfstring{\la}{Lambda}-Rings and power structures}
\label{lambda}
Let
$$\La=\varprojlim\cZ\pser{x_1,\dots,x_n}^{S_n}$$
be the ring of symmetric functions \cite{macdonald_symmetric}. It is well-known that \La is generated as an algebra over $\cZ$ by elementary symmetric functions
$$e_n=\sum_{i_1<\dots<i_n}x_{i_1}\dots x_{i_n}$$
as well as by complete symmetric functions
$$h_n=\sum_{i_1\le\dots\le i_n}x_{i_1}\dots x_{i_n}.$$
Moreover $\La_\cQ=\La\ts\cQ$ is generated over $\cQ$ by power sums $p_n=\sum x_i^n$.

A $\cQ$-algebra $R$ is called a \la-ring if it is endowed with a map $\circ:\La\xx R\to R$ called plethysm, such that $(-\circ r):\La\to R$ is a ring homomorphism for any $r\in R$ and the maps $\psi_n=(p_n\circ-)$, called Adams operations, are ring homomorphisms satisfying $\psi_1=\Id_R$ and $\psi_m\psi_n=\psi_{mn}$ for $m,n\ge1$. Note that plethysm is uniquely determined by Adams operations. It is also uniquely determined by maps $\la_n=(e_n\circ-):R\to R$ called \la-operations and by maps $\si_n=(h_n\circ-):R\to R$ called \si-operations.

Given a \la-ring $R$, we endow the ring $A=R\pser{x_1,\dots,x_m}$ with a \la-ring structure by the rule
$$\psi_n(rx^\al)=\psi_n(r)x^{n\al},\qquad r\in R,\al\in\cN^{m}.$$
Let $A_+\sb A$ be an ideal generated by $x_1,\dots,x_m$. We define a map $\Exp:A_+\to 1+A_+$, called plethystic exponential, by the rule \cite{getzler_mixed,mozgovoy_computational}
$$\Exp(f)=\sum_{n\ge0}\si_n(f)=\exp\bigg(\sum_{n\ge1}\frac1n\psi_n(f)\bigg).$$
This map has an inverse $\Log:1+A_+\to A_+$, called plethystic logarithm,
$$\Log(f)=\sum_{n\ge1}\frac{\mu(n)}{n}\psi_n\log(f),$$
where $\mu$ is a M\"obius function.

We define a power structure map $\Pow:(1+A_+)\xx A\to 1+A_+$ by the rule~\cite{mozgovoy_computational}
$$\Pow(f,g)=\Exp(g\Log(f)).$$
In the case when $R$ is a ring of motives, the power structure map has the following geometric interpretation \cite{gusein-zade_powerb}. Let
$$f=1+\sum_{\al>0}[A_\al]x^\al,$$
where $A_\al$ are algebraic varieties. Then
$$\Pow(f,[X])=\sum_{k:\cN^m\to\cN}\bigg[\Big(F_{\n k}X\xx\prod_{\al\in\cN^m}A_{\al}^{k(\al)}\Big)/\prod_{\al\in\cN^m}S_{k(\al)}\bigg]x^{\sum k(\al)\al},$$
where the sum runs over maps $k:\cN^m\to\cN$ with finite support, $\n k=\sum_{\al\in\cN^m}k(\al)$, the configuration space $F_nX$ is given by
$$F_nX=\sets{(x_1,\dots,x_n)\in X^n}{x_i\ne x_j\text{ for }i\ne j},$$
and the product of symmetric groups $\prod_{\al\in\cN^m}S_{k(\al)}$ acts on both factors in the obvious way. The quotient in square brackets parametrizes elements in
$$\bigcup_{\psi:X\to\cN^m}\prod_{x\in X}A_{\psi(x)}$$
with $\psi:X\to\cN^m$ satisfying $\#\sets{x\in X}{\psi(x)=\al}=k(\al)$ for any $\al\in\cN^m\ms\set0$ (see \cite{mozgovoy_motivicb}). Therefore we can also write
$$\Pow(f,[X])=\sum_{\psi:X\to\cN^m}\prod_{x\in X}[A_{\psi(x)}]x^{\psi(x)},$$
where the sum runs over maps $\psi:X\to\cN^m$ with finite support.

\subsection{Quivers and moduli spaces}
Let $Q$ be a quiver, with vertex set $Q_0$ and edge set $Q_1$.
For an arrow $a\in Q_1$, we denote by $s(a)\in Q_0$ (resp.\ $t(a)\in Q_0$) the vertex at which $a$ starts (resp.\ ends).
We define the Euler-Ringel form $\chi$ on $\cZ^{Q_0}$ by the rule
$$\chi(\al,\be)=\sum_{i\in Q_0}\al_i\be_i-\sum_{a\in Q_1}\al_{s(a)}\be_{t(a)},\qquad \al,\be\in\cZ^{Q_0}.$$
We define the skew-symmetric bilinear form $\ang{\bullet,\bullet}$ of the quiver $Q$ to be 
$$\ang{\al,\be}=\chi(\al,\be)-\chi(\be,\al),\qquad \al,\be\in\cZ^{Q_0}.$$

Given a $Q$-representation $M$, we define its dimension vector $\udim M\in\cN^{Q_0}$ by $\udim M=(\dim M_i)_{i\in Q_0}$.
Let $\al\in\cN^{Q_0}$ be a dimension vector and let $V_i=\cC^{\al_i}$, $i\in Q_0$.
We define
$$R(Q,\al)=\bigoplus_{a\in {Q_1}}\Hom(V_{s(a)},V_{t(a)})$$
and 
\[\GG_\al=\prod_{i\in Q_0}\GL(V_i).\]
Note that $\GG_\al$ naturally acts on $R(Q,\al)$ and the quotient stack
$$\gM(Q,\al)=[R(Q,\al)/\GG_\al]$$
gives the moduli stack  of representations of $Q$ with dimension vector $\al$.

Let $W$ be a potential on $Q$, a finite linear combination of cyclic paths in $Q$. Denote by $J=J_{Q,W}$ 
the Jacobian algebra, the quotient of the path algebra $\C Q$ by the two-sided ideal generated by formal partial
derivatives of the potential $W$. Let
$$f_\al:R(Q,\al)\to\cC$$ be the $\GG_\al$-invariant function defined by taking the trace of the map associated 
to the potential $W$. As it is now well known~\cite[Proposition 3.8]{segal_a}, a point in the critical locus 
$\crit(f_\al)$ corresponds to a $J$-module. The quotient stack 
$$\gM(J, \al)=\bigl[\crit(f_\al)/\GG_\al\bigr]$$
gives the moduli stack of $J$-modules with dimension vector $\al$.

\begin{defn}
\label{central charge}
A {\it central charge} is a group homomorphism $Z:\cZ^{Q_0}\to \C$
such that $$Z(\al)\in\cH_+=\sets{re^{i\pi\vi}}{r>0,0<\vi\le1}$$
for any $\al\in \cN^{Q_0}\ms\set0$. 
Given $\al\in \cN^{Q_0}\ms\set0$, the number $\vi(\al)=\vi\in(0,1]$ such that 
$Z(\al)=re^{i\pi\vi}$, for some $r>0$, is called the phase of \al.
\end{defn}

\begin{defn}
\label{semist}
For any nonzero $Q$-representation (resp.\ $J$-module) $V$, we define $\vi(V)=\vi(\udim V)$.
A $Q$-representation (resp.\ $J$-module) $V$ is said to be $Z$-(semi)stable if for any proper nonzero $Q$-subrepresentation (resp.\ $J$-submodule) $U\sb V$ we have
$$\vi(U)(\le)\vi(V).$$
\end{defn}

\begin{defn}
\label{Z from ze}
Given $\zeta\in\cR^{Q_0}$, define the central charge $Z:\cZ^{Q_0}\to\cC$ by the rule
$$Z(\al)=-\zeta\cdot\al+i\n\al,$$
where $\n\al=\sum_{i\in Q_0}\al_i$.
We say that a $Q$-representation (resp.\ $J$-module) is \ze-(semi)stable if it is $Z$-(semi)stable.
\end{defn}

\begin{rk}
\label{rmr:mu}
Let the central charge $Z$ be as in Definition \ref{Z from ze}.
Define the slope function $\mu:\cN^{Q_0}\ms\set0\to\cR$ by $\mu(\al)=\frac{\ze\cdot\al}{\n\al}$. If $l\sb\cH=\cH_+\cup\set0$ is a ray such that $Z(\al)\in l$ then $l=\cR_{\ge0}(-\mu(\al),1)$. This implies that $\vi(\al)<\vi(\be)$ if and only if $\mu(\al)<\mu(\be)$.
\end{rk}

We say that $\ze\in\cR^{Q_0}$ is \al-generic if for any $0<\be<\al$ we have $\vi(\be)\ne\vi(\al)$.
This condition implies that any \ze-semistable $Q$-representation (resp.\ $J$-module) is automatically \ze-stable.

Let $R_\ze(Q,\al)$ denote the open subset of $R(Q,\al)$ consisting of \ze-semistable representations.
Let $f_{\ze,\al}$ denote the restriction of $f_{\al}$ to $R_\ze(Q,\al)$. 
The quotient stacks
\begin{equation}
\gM_\ze(Q,\al)=\bigl[R_\ze(Q,\al)/\GG_\al\bigl],\qquad
\gM_\ze(J, \al)=\bigl[\crit(f_{\ze,\al})/\GG_\al\bigr]
\label{eq:def moduli stacks}
\end{equation} 
give the moduli stacks of $Q$-representations and $J$-modules with dimension vector~\al.

\subsection{Motivic DT invariants}
Let $(Q,W)$ be a quiver with a potential and let $J=J_{Q,W}$ be its Jacobian algebra. Recall that the degeneracy locus of the function $f_\al:R(Q,\al)\to\cC$
defines the locus of $J$-modules, so that the quotient stack $$\gM(J,\al)=[\crit(f_\al)/\GG_\al]$$
is the stack of $J$-modules with dimension vector~$\al$. 
We define motivic Donaldson-Thomas invariants
$$[\gM(J,\al)]_\vir=\frac{[\crit(f_\al)]_\vir}{[\GG_\al]_\vir},$$
where $[\GG_\al]_\vir$ refers to the virtual motive of the pair $(\GG_\al,0)$. 

\begin{defn}
A subset $I\subset Q_1$ is called a cut of $(Q,W)$ if in the associated grading $g_I$ on $Q$ given by
$$g_I(a) = 
\begin{cases}
1 & a \in I,\\
0 & a \notin I,
\end{cases}$$
the potential $W$ is homogeneous of degree $1$.
\end{defn}

Throughout this section we assume that $(Q,W)$ admits a cut. Then the space $R(Q,\al)$ admits a $\cC^*$-action satisfying the conditions of Theorem \ref{thr:bbs} for the function $f_\al:R(Q,\al)\to\cC$. This implies
\begin{multline}
\label{eq:bbs1}
[\gM(J,\al)]_\vir
=(-\cL^\oh)^{-\dim R(Q,\al)}\frac{[f_\al\inv(0)]-[f_\al\inv(1)]}{[\GG_\al]_\vir}\\
=(-\cL^\oh)^{\hi(\al,\al)}\frac{[f_\al\inv(0)]-[f_\al\inv(1)]}{[\GG_\al]}.
\end{multline}
Generally, for an arbitrary stability parameter $\ze$, we define
\begin{equation}
[\gM_\ze(J,\al)]_\vir
=(-\cL^\oh)^{\hi(\al,\al)}\frac{[f^{-1}_{\ze,\al}(0)]-[f^{-1}_{\ze,\al}(1)]}{[\GG_\al]},
\label{eq:bbs2}
\end{equation}
where, as before, $f_{\ze,\al}$ denote the restriction of $f_{\al}:R(Q,\al)\to\cC$ to $R_\ze(Q,\al)$.

\begin{lem}
Let $\al\in\cN^{Q_0}$ be such that $\al_i=1$ for some $i\in Q_0$ (this will be the case for framed representations studied later) and let $\ze\in\cR^{Q_0}$ be \al-generic. Then
$$
[\gM_\ze(J,\al)]_\vir=\frac{[\crit(f_{\ze_,\al})]_\vir}{[\GG_\al]_\vir}.
$$
\end{lem}
\begin{proof}
Let
$$M_\ze(Q,\al)=R_\ze(Q,\al)/\GG_\al$$
be the smooth moduli space of \ze-semistable $Q$-representations having dimension vector \al, and let $f'_{\ze,\al}:M_\ze(Q,\al)\to\cC$ be the map induced by $f_{\ze,\al}:R_\ze(Q,\al)\to\cC$. Note that $R_\ze(Q,\al)\to M_\ze(Q,\al)$ is a principal bundle with the structure group $PG_\al=G_\al/\cC^*$. The group $PG_\al$ is a product of general linear groups (here we use our assumption that there exists $i\in Q_0$ with $\al_i=1$). Therefore $R_\ze(Q,\al)\to M_\ze(Q,\al)$ is locally trivial in Zariski topology. This implies
$$\frac{[\crit(f_{\ze_,\al})]_\vir}{[\GG_\al]_\vir}
=\frac{[\crit(f'_{\ze,\al})]_\vir}{[\GL_1]_\vir}.$$
As $(Q,W)$ admits a cut, the space $M_\ze(Q,\al)$ admits a $\cC^*$-action satisfying the conditions of Theorem \ref{thr:bbs} for the function $f'_{\ze,\al}:M_\ze(Q,\al)\to\cC$ (one uses the fact that $M_\ze(Q,\al)$ is projective over $R(Q,\al)/\!\!/\GG_\al$).
This implies
\begin{multline}
\frac{[\crit(f'_{\ze,\al})]_\vir}{[\GL_1]_\vir}
=\frac{-(-\cL^\oh)^{-\dim M_\ze(Q,\al)}[\vi_{f'_{\ze,\al}}]}{(-\cL^\oh)\inv(\cL-1)}\\
=\frac{(-\cL^\oh)^{\dim \GG_\al-\dim R(Q,\al)}}{\cL-1}([f'^{-1}_{\ze,\al}(0)]-[f'^{-1}_{\ze,\al}(1)])\\
=(-\cL^\oh)^{\hi(\al,\al)}\frac{[f^{-1}_{\ze,\al}(0)]-[f^{-1}_{\ze,\al}(1)]}{[\GG_\al]}.
\end{multline}
\end{proof}

\subsection{Twisted algebra and central charge}

\begin{defn}
\label{defn_mt}
The {\it twisted motivic algebra} associated to the quiver $Q$ is the associative $\tildeMC$-algebra 
\[
\That_Q = \prod_{\al\in \cN^{Q_0}} \wtl{\lM}_\cC\cdot y^\al\]
generated by formal variables $y^\al$ that satisfy the relation
$$y^{\al}\cdot y^{\be}=(-\cL^\oh)^{\ang{\al,\be}}y^{\al+\be},$$
with $\langle\bullet,\bullet \rangle$ the skew-symmetric form of the quiver $Q$.
\end{defn}

Note that if the quiver $Q$ is symmetric, i.e.~its skew-symmetric form is identically zero, then $\That_Q$ is commutative. 

\begin{rk}
This algebra, which is all we are going to need, is (a completion of) the ``positive half'' of the motivic quantum torus of Kontsevich--Soibelman~\cite{kontsevich_stability}. 
\end{rk}

A ray in the upper half plane $\cH=\cH_+\cup\set0$ is a half line which has the origin as its end.
For a ray $l\subset\cH$ and a central charge $Z$, we put
\[
\That_{Z,l} =
\prod_{\al\in Z^{-1}(l)\cap \cN^{Q_0}} \wtl{\lM}_\cC\cdot y^\al,
\]
a subalgebra of the twisted algebra $\That_Q$. 

\begin{lem}[Kontsevich--Soibelman~{\cite[Theorem 6]{kontsevich_affine}}]
For any element $A=\sum_{\al\in\cN^{Q_0}}A_\al y^\al\in \That_Q$ with $A_0=1$, there is a unique factorization
\begin{equation}\label{eq_FP}
A
=\prod_{l\subset\cH}^{\curvearrowright}A_{Z,l}
\end{equation}
with $A_{Z,l}\in \That_{Z,l}$, where the product is taken in the clockwise order over all rays.
\label{factorlemma}
\end{lem}
\begin{proof}
For a positive real number $r$, we put
\[
\That^{(r)}_Q=
\prod_{\al\in \cN^{Q_0}, |\al|<r} \wtl{\lM}_\cC\cdot y^\al
,\quad
\That_{Z,l}^{(r)} =
\prod_{\al\in Z^{-1}(l)\cap \cN^{Q_0}, |\al|<r} \wtl{\lM}_\cC\cdot y^\al
\]
which we consider as factor algebras of $\That_Q$ and $\That_{Z,l}$ respectively.
Let $A^{(r)}\in \That^{(r)}_Q$ denote the image of $A$
under the canonical projection $\That_{Q}\twoheadrightarrow\That^{(r)}_Q$.
It is enough to show that for any $r$ there is a unique factorization
\[
A^{(r)}
=\prod_{l\subset\cH}^{\curvearrowright}A_{Z,l}^{(r)}
\]
with $A_{Z,l}^{(r)}\in \That_{Z,l}^{(r)}$.
Note that the set of rays $l\in\cH$ such that
$$\sets{\al\in Z^{-1}(l)}{\n\al<r}\neq \emptyset$$
is finite. We order this set $(l_1,\ldots,l_N)$ so that
$$\arg Z(l_1)<\cdots<\arg Z(l_N).$$
First we put $A_{Z,l_1}^{(r)}$ to be the summand of $A^{(r)}$ contained in
$\That_{Z,l_1}^{(r)}$.
For $1<i\leq N$ we define $A_{Z,l_{i}}^{(r)}$ to be the summand 
of
\[
A^{(r)}\cdot \bigl(A_{Z,l_1}^{(r)}\bigr)^{-1}\cdot \cdots \cdot
\bigl(A_{Z,l_{i-1}}^{(r)}\bigr)^{-1}
\]
contained in
$\That_{Z,l_i}^{(r)}$.
Uniqueness is clear from the construction.
\end{proof}

\subsection{Generating series of motivic DT invariants}
Let $(Q,W)$ be a quiver with a potential admitting a cut, and let $J=J_{Q,W}$ be its Jacobian algebra.

\begin{defn}
\label{def:series}
We define the generating series of the motivic Donaldson-Thomas invariants of $(Q,W)$ by
$$A_U
=\sum_{\al\in\cN^{Q_0}}[\gM(J,\al)]_\vir \cdot y^\al
=\sum_{\al\in\cN^{Q_0}}
\frac{[\crit(f_\al)]_\vir}{[\GG_\al]_\vir}\cdot y^\al\in\That_Q,$$
the subscript referring to the fact that we think of this series as the universal series.
\end{defn}

Given a cut $I$ of $(Q,W)$, we define a new quiver $Q_I=(Q_0,Q_1\ms I)$. 
Let $J_{W,I}$ be the quotient of $\cC Q_I$ by the ideal $$(\dd_I W)=(\dd W/\dd a,a\in I).$$

\begin{prop}
\label{first reduction}
If $(Q,W)$ admits a cut $I$, then
$$A_U=\sum_{\al\in\cN^{Q_0}}(-\cL^\oh)^{\hi(\al,\al)+2d_I(\al)}\frac{[R(J_{W,I},\al)]}{[\GG_\al]}y^\al,$$
where $d_I(\al)=\sum_{(a:i\to j)\in I}\al_i\al_j$ for any $\al\in\cZ^{Q_0}$.
\end{prop}
\begin{proof}
Let $f=f_\al:R(Q,\al)\to\cC$.
According to \eqref{eq:bbs1} we have
$$[\gM(J,\al)]_\vir=(-\cL^\oh)^{\hi(\al,\al)}\frac{[f\inv(0)]-[f\inv(1)]}{\GG_\al}.$$
It is proved in \cite[Theorem 4.1]{nagao_wall-crossing} and \cite[Prop.~7.1]{morrison_motivic} that
$$[f\inv(1)]-[f\inv(0)]=-\cL^{d_I(\al)}[R(J_{W,I},\al)].$$
Therefore
$$
[\gM(J,\al)]_\vir
=(-\cL^\oh)^{\hi(\al,\al)+2d_I(\al)}\frac{[R(J_{W,I},\al)]}{[\GG_\al]}.$$
\end{proof}

Let $\ze\in\cR^{Q_0}$ be some stability parameter and let $Z:\cZ^{Q_0}\to\cC$ be the central charge determined by \ze as in Definition \ref{Z from ze}.



\begin{defn}
\label{A_{Z,l}}
Let $l=\cR_{\ge0}(-\mu,1)\sb\cH$ be a ray (see Remark \ref{rmr:mu}).
We put
$$A_{Z,l}=A_{\ze,\mu}=
\sum_{\substack{\al\in \cN^{Q_0} \\ Z(\al)\in l}}
[\gM_\ze(J,\al)]_\vir\cdot y^\al
\in \That_Q.$$
\end{defn}
The Harder-Narashimhan filtrations provide a filtration on $R(Q,\al)$.
This filtration induces the following {\it factorization property}.

\begin{thm}
\label{thm_FP} 
Assume that $(Q,W)$ has a cut. Then we have
\[
A_U=\prod_{l}^{\curvearrowright}A_{Z,l},
\]
where the product is taken in the clockwise order over all rays.
\end{thm}
\begin{proof}
This is originally a result of Kontsevich--Soibelman~\cite{kontsevich_stability}, though their proof 
depends on a conjectural integral identity. Assuming the existence of a cut, Theorem \ref{thr:bbs}
leads to a simplified proof, written out in~\cite{nagao_wall-crossing} and~\cite{mozgovoy_motivica}. 
\end{proof}


\section{The universal DT series of the conifold quiver}

\subsection{Motivic DT invariants for the conifold quiver}
\label{sec:coniquiver}
Let $(Q,W)$ be the conifold quiver with potential. 
Recall that $Q$ has vertices $0,1$ and arrows $a_i:0\to1$, $b_i:1\to0$ for $i=1,2$.
The potential is given by
$$W=a_1b_1a_2b_2-a_1b_2a_2b_1.$$
We have
$$\hi(\al,\al)=\al_0^2+\al_1^2-4\al_0\al_1.$$


Let $J_W=\cC Q/\dd W$ be the Jacobian algebra of $(Q,W)$.
Then $I=\set{a_1}$ is easily seen to be a cut for $(Q,W)$.
Let $Q_I=(Q_0,Q_1\ms I)$ be the quiver defined by the cut, and  $J_{W,I}$ 
the quotient of $\cC Q_I$ by the ideal $$(\dd_I W)=(\dd W/\dd a,a\in I).$$
It follows from Proposition \ref{first reduction} that the coefficients of the universal Donaldson-Thomas series $A_{U}=\sum_{\al\in\cN^{Q_0}} A_\al y^\al$ are given by
$$A_\al
=(-\cL^\oh)^{\hi(\al,\al)+2\al_0\al_1}\frac{[R(J_{W,I},\al)]}{[\GG_\al]}\\
=(-\cL^\oh)^{(\al_0-\al_1)^2}\frac{[R(J_{W,I},\al)]}{[\GG_\al]}.$$

The goal of this section is to prove the following result.

\begin{thm}
\label{thm main universal result}
We have
\begin{equation} 
A_{U}(y_0, y_1) =\displaystyle\Exp\bigg(\frac{(\cL+\cL^2)y_0y_1-\cL^\oh(y_0+y_1)}{\cL-1}\sum_{n\ge0}(y_0y_1)^n\bigg).
\label{eq univ exp form}
\end{equation}
Equivalently,
\begin{equation}
A_{U}(y_0, y_1) =\displaystyle\prod_{\alpha\in\Delta_+}A^\alpha(y_0, y_1),\label{eq univ prod form}
\end{equation}\
where for roots $\alpha\in\Delta_+$, we put
\[
A^\alpha(y_0,y_1)=
\begin{cases}
\Exp\bigg(\displaystyle\frac{-\cL^{-\oh}}{1-\cL\inv}y^\al\bigg)=
\displaystyle\prod_{j\ge0}
\left(1-\cL^{-j-\oh}y^\al\right) &\alpha\in\De_+^\re,\\
\Exp\left(\displaystyle\frac{1+\cL}{1-\cL\inv}y^\al\right)=
\displaystyle\prod_{j\ge0}
\left(1-\cL^{-j}y^\al\right)^{-1}
\left(1-\cL^{-j+1}y^\al\right)^{-1} &\alpha\in\De_+^\im.\\
\end{cases}
\]
\end{thm}

The equivalence of the exponential and product forms \eqref{eq univ exp form}--\eqref{eq univ prod form}
follows from formal manipulations. 
In the following two subsections, we give two proofs of Theorem~\ref{thm main universal result}. 
The first one develops the method of~\cite{feit_pairs} (c.f.~\cite{bryan_motivic}). The second proof uses another ``dimensional reduction'' 
to reduce the problem to that of representations of the tame quiver of affine type $A_1^{(1)}$.

\subsection{First proof}
The goal is to compute the generating function of motives of moduli of representations of $J_{W,I}$-modules. Up to a group action, these moduli spaces are given concretely as spaces of triples of matrices:
\[ R(J_{W,I},\alpha) =\{ (A_2,B_1,B_2)\in\Hom(V_0,V_1)\times\Hom(V_1,V_0)^{\times 2}\mid B_1A_2B_2 = B_2A_2B_1\}.\]
In this section, for simplicity, we will denote this space by $R(\alpha)$.
The proof begins by reducing the problem to two simpler ones via a stratification of $R(\alpha)$.
For $(A_2,B_1,B_2)\in R(\alpha)$, consider the linear map
\[ A_2\oplus B_2 : V_0\oplus V_1\to V_0\oplus V_1 .\]
For any such endomorphism, the vector space $V= V_0\oplus V_1$ has a second decomposition $V=V^I\oplus V^N$ on which $A_2\oplus B_2$ decomposes into an invertible map and a nilpotent map (c.f.~\cite[Lemma 1]{feit_pairs}). Namely,
\[ A_2^I\oplus B_2^I : V^I\to V^I\textrm{ and } A_2^N\oplus B_2^N : V^N\to V^N .\]
Define $V_i^I$ (resp. $V_i^N$) to be the intersection of $V_i$ and $V^I$ (resp. $V^N$), so that we have decompositions $V_0 = V_0^I\oplus V_0^N$ and $V_1 = V_1^I\oplus V_1^N$, with
\[A_2 = A_2^I\oplus A_2^N\in\Hom(V_0^I,V_1^I)\oplus\Hom(V_0^N,V_1^N),\]
\[B_2 = B_2^I\oplus B_2^N\in\Hom(V_1^I,V_0^I)\oplus\Hom(V_1^N,V_0^N).\]
Notice that $A_2^I,B_2^I$ are invertible, in particular we have 
\[\dim(V_0^I) =\dim(V_1^I) =\oh\dim(V^I).\]
A little bit of linear algebra shows that a matrix $B_1$ satisfying $B_1A_2B_2 = B_2A_2B_1$ has a similar block decomposition with respect to the splitting $V=V^I\oplus V^N$;
\[ B_1 = B_1^I\oplus B_1^N\in\Hom(V_1^I,V_0^I)\oplus\Hom(V_1^N,V_0^N).\]
The relation $B_1A_2B_2 = B_2A_2B_1$ now becomes two independent sets of equations, 
$B_1^IA_2^IB_2^I = B_2^IA_2^IB_1^I$ and $B_1^NA_2^NB_2^N = B_2^NA_2^NB_1^N$. Define
\[R_a^I=\{ (A_2,B_1,B_2)\in R((a,a))\mid A_2\oplus B_2\textrm{ is invertible}\};\]
\[R_{\alpha}^N =\{ (A_2,B_1,B_2)\in R(\alpha)\mid A_2\oplus B_2\textrm{ is nilpotent}\}.\]
Over the stratum of $R(\alpha)$ where $\dim(V^I) = 2a$, we have a Zariski locally trivial fibre bundle
\begin{gather*}
\xymatrix{
R_a^I\times R_{\alpha}^N\ar[r] &\{ (A_2,B_1,B_2)\in R(\alpha)\mid\dim(V_I)=2a\}\ar[d]\\
 & \mathcal{M}(a,\alpha)}
\end{gather*}
where $\mathcal{M}(a,\alpha)$ is the space of direct sum decompositions $V_0\cong V^I_0\oplus V_0^N$, $V_1\cong V^I_1\oplus V_1^N$. Hence stratifying $R(\alpha)$ by $\dim(V^I)$ gives the following relation in the Grothendieck ring of varieties:
\[ [R(\alpha)] =\sum_{a=0}^{\min (\alpha_0,\alpha_1)} [R_a^I]\cdot [R^N_{(\alpha_0-a,\alpha_1-a)}]\cdot 
\frac{[\GL_{\alpha_0}]}{[\GL_a][\GL_{\alpha_0-a}]}\cdot\frac{[\GL_{\alpha_1}]}{[\GL_a][\GL_{\alpha_1-a}]}.\]
We collect the above motives into two generating series
\[ I(y) =\sum_{a\geq 0 }\frac{[R_a^I]}{[\GL_a]^2} y^a\]
and
\[N(y_0,y_1) =\sum_{\alpha\in N^Q} 
\frac{[R^N_{\alpha}]}{[\GL_{\alpha_0}][\GL_{\alpha_1}]}(-\LL^{1/2})^{(\alpha_0-\alpha_1)^2}y_0^{\alpha_0}y_1^{\alpha_1} .\]
Multiplying the above relation by $(-\LL^{1/2})^{(\alpha_0-\alpha_1)^2}t_0^{\alpha_0}t_1^{\alpha_1}$ and summing 
gives an equality of power series
\[ A_{U} = I( y_0 y_1 )\cdot N( y_0,y_1 ).\]
It remains to compute $I(y)$ and $N(y_0,y_1)$.

First consider $I(y)$. If $\pi$ is a partition of $ a $ we will write $\pi\vdash a$, and denote its length $l(\pi)$, and size $|\pi|$. Then we have two spaces
\[
R_a^I = \{ (A_2,B_1,B_2)\in\textrm{Iso}\!\left(V_0^I,V_1^I\right)\times\Hom\!\left( V_1^I,V_0^I\right)\times\textrm{Iso}\!\left(V_1^I,V_0^I\right)\mid B_1A_2B_2 = B_2A_2B_1\}\]
and
\[C_a^I =\{ (C_1,C_2)\in \End (V_1^I) \times\GL (V_1^I)\mid C_2^{-1} C_1 C_2 = C_1\},\]
together with a map $\beta : R^I_a\to C^I_a$ given by
\[
\beta(A_2,B_1,B_2) = (A_2B_1,A_2B_2).
\]
The map $\beta$ is a $\GL (V_0^{I})$-torsor associated to a global gauge fixing, $g\cdot (A_2,B_1,B_2) = (A_2 g^{-1},gB_1,gB_2) $. Since the general linear group is a special group, the map $\beta$ is a locally trivial $\GL (V_0^I)$ bundle in the Zariski topology. The base of the fibration $C^I_a$ is a commuting variety whose motivic class is known~\cite{bryan_motivic} to equal
\[ [\GL_a]\sum_{\pi\vdash a}\LL^{l(\pi)}.\]
Therefore
\begin{eqnarray*}
 I(y) & = &\sum_{a\geq 0 }\frac{[R_a^I]}{[\GL_a]^2}y^a =\sum_{a\geq 0 }\frac{[C_a^I]}{[\GL_a]} y^a =\sum_{\pi}\LL^{l(\pi)}y^{|\pi|}  \\
      & = &\prod_{i=1}^\infty\frac{1}{1-\LL y^i}=\Exp\bigg(\LL\sum_{n\geq 1} y^n\bigg).
\end{eqnarray*}
All that remains is to compute the series $N(y_0,y_1)$. Given now that the matrix $ A_2\oplus B_2 $ is nilpotent, there 
exists a basis of $V^N$, $\{ a^{i_1}_{s_1},b^{i_2}_{s_2},c^{i_3}_{s_3},d^{i_4}_{s_4}\}$, $1\leq i_j 
\leq k_j$, $ 1\leq s_j\leq r_i^j $, such that
\begin{eqnarray*}
V_0^N \ni a^{i_1}_{r^1_i}\overset{A_2}\mapsto a^{i_1}_{r^1_i-1}\overset{B_2}\mapsto a^{i_1}_{r^1_i-2}\overset{A_2}\mapsto\cdots\overset{A_2}\mapsto a^{i_1}_{1}\in V_1^N\setminus 0\textrm{ and } B_2(a^{i_1}_{1}) =0,\\
V_1^N \ni b^{i_2}_{r^2_i}\overset{B_2}\mapsto b^{i_2}_{r^2_i-1}\overset{A_2}\mapsto b^{i_2}_{r^2_i-2}\overset{B_2}\mapsto\cdots\overset{B_2}\mapsto b^{i_2}_{1}\in V_0^N\setminus 0\textrm{ and } A_2(b^{i_2}_{1}) =0,\\
V_0^N \ni c^{i_3}_{r^3_i}\overset{A_2}\mapsto c^{i_3}_{r^3_i-1}\overset{B_2}\mapsto c^{i_3}_{r^3_i-2}\overset{A_2}\mapsto\cdots\overset{B_2}\mapsto c^{i_3}_{1}\in V_0^N\setminus 0\textrm{ and } A_2(c^{i_3}_{1}) =0,\\
V_1^N \ni d^{i_4}_{r^4_i}\overset{B_2}\mapsto d^{i_4}_{r^4_i-1}\overset{A_2}\mapsto d^{i_4}_{r^4_i-2}\overset{B_2}\mapsto\cdots\overset{A_2}\mapsto d^{i_4}_{1}\in V_1^N\setminus 0\textrm{ and } B_2(d^{i_4}_{1}) =0.
\end{eqnarray*}

As the numbers $r^1_{i},r^2_i$ are always even and $r^3_i,r^4_i$ always odd, it is combinatorially convenient to define 
$\hat{r}^1_{i}=r^1_{i}/2 ,\hat{r}^2_{i}=r^2_{i}/2  ,\hat{r}^3_{i}=(r^3_{i}+1)/2 ,\hat{r}^4_{i}=(r^4_{i}+1)/2$. Also after reordering we may assume that $\hat{r}^j_{i}\geq\hat{r}_{i+1}^j$. Up to a choice of the above basis, the matrix $ A_2\oplus B_2 $ is determined by four partitions

\[\pi_1 : |\pi_1| =\hat{r}^1_{1}+\hat{r}^1_{2}+\hat{r}^1_{3}+\ldots +\hat{r}^1_{k_1}\]
\[\pi_2 : |\pi_2| =\hat{r}^2_{1}+\hat{r}^2_{2}+\hat{r}^2_{3}+\ldots +\hat{r}^2_{k_2}\]
\[\pi_3 : |\pi_3| =\hat{r}^3_{1}+\hat{r}^3_{2}+\hat{r}^3_{3}+\ldots +\hat{r}^3_{k_3}\]
\[\pi_4 : |\pi_4| =\hat{r}^4_{1}+\hat{r}^4_{2}+\hat{r}^4_{3}+\ldots +\hat{r}^4_{k_4}\]
with
\[\alpha_0 = |\pi_1|+|\pi_2|+|\pi_3|+|\pi_4|-l(\pi_4)\]
\[\alpha_1 = |\pi_1|+|\pi_2|+|\pi_3|+|\pi_4|-l(\pi_3).\]
With respect to the above basis, denote the normal form of $A_2\oplus B_2$ by $A_2^{\{\pi_j\}}\oplus B_2^{\{\pi_j\}}$. 
The space $R_{\alpha}^N$ can be stratified by this data, giving
\[ [R_{\alpha}^N] =\sum_{\pi_1,\pi_2,\pi_3,\pi_4} [R(\pi_1,\pi_2,\pi_3,\pi_4)],\]
where $R(\pi_1,\pi_2,\pi_3,\pi_4)$ is the stratum of $R_{\alpha}^N$, where $A_2\oplus B_2$ has normal form 
$A_2^{\{\pi_j\}}\oplus B_2^{\{\pi_j\}}$. The space $R(\pi_1,\pi_2,\pi_3,\pi_4)$ is a vector bundle
\[ p: R(\pi_1,\pi_2,\pi_3,\pi_4)\to \{(A_2,B_2)\mid A_2\oplus B_2\sim A_2^{\{\pi_j\}}\oplus B_2^{\{\pi_j\}}\}\]
over the space of all matrices with this normal form, with fibre the linear space of matrices
\[\{B_1\mid B_1 A_2^{\{\pi_j\}} B_2^{\{\pi_j\}} = B_2^{\{\pi_j\}} A_2^{\{\pi_j\}} B_1\}.\]
We compute the fibre and base by a linear algebra calculation to deduce
\[ [R(\pi_1,\pi_2,\pi_3,\pi_4)] = [\GL_{\alpha_0}]\cdot [\GL_{\alpha_1}] f(\pi_1) f(\pi_2) g(\pi_3) g(\pi_4) (-\LL^{1/2})^{-(l(\pi_3)-l(\pi_4) )^2},\]
where we are given
\[  f(\pi ) =\prod_{i\geq 1}\LL^{b_i^2}/[\GL_{b_i}]\textrm{ for }  \pi=(1^{b_1}2^{b_2}3^{b_3}\cdots), \]
and
\[ g(\pi ) =\prod_{i\geq 1} (-\LL^{1/2})^{b_i^2}/[\GL_{b_i}]\textrm{ for }  \pi=(1^{b_1}2^{b_2}3^{b_3}\cdots) .\]
Substituting this into the generating series gives

\[\begin{array}{rcl} N(y_0,y_1) & = &\displaystyle\sum_{\alpha_0,\alpha_1\geq 0}\sum_{\substack{\pi_1,\pi_2,\pi_3,\pi_4\\ \alpha_0 = |\pi_1|+|\pi_2|+|\pi_3|+|\pi_4|-l(\pi_4)\\ \alpha_1 = |\pi_1|+|\pi_2|+|\pi_3|+|\pi_4|-l(\pi_3) } } f(\pi_1)f(\pi_2)g(\pi_3)g(\pi_4)y_0^{\alpha_0}y_1^{\alpha_1}\\
 & = & \displaystyle\sum_{\pi_1}f(\pi_1)(y_0y_1)^{|\pi_1|} \displaystyle\sum_{\pi_2}f(\pi_2)(y_0y_1)^{|\pi_2|}\displaystyle\sum_{\pi_3}g(\pi_3)(y_0y_1)^{|\pi_3|}y_1^{-l(\pi_3)}\\
 & &\cdot \displaystyle\sum_{\pi_4} g(\pi_4)(y_0y_1)^{|\pi_4|}y_0^{-l(\pi_4)}.\end{array}\]

The series for $f$ and $g$ have well know formulas \cite{macdonald_symmetric}
\[\sum_{\pi} f(\pi)y^{|\pi|} =\prod_{i,j=1}^{\infty} (1-\LL^{1-j}y^i)^{-1} =\Exp\bigg(\frac{\LL}{\LL -1}\sum_{n\geq 1} y^n\bigg),\] 
and 
\[\sum_{\pi} g(\pi)y^{|\pi|}a^{-l(\pi)} =\prod_{i,j=1}^{\infty} (1+(-\LL^{1/2})^{-2j+1}y^{i}a^{-1}) = \Exp\bigg(\frac{\LL^{1/2}}{1-\LL}\sum_{n\geq 1} y^n a^{-1}\bigg).\]
Hence
\[N(y_0,y_1) =\Exp\bigg(\frac{2\LL}{\LL -1}\sum_{n\geq 1} (y_0y_1)^n\bigg)\cdot\Exp\bigg(\frac{-\LL^{1/2}}{\LL-1}\sum_{n\geq 1} y_0^ny_1^{n-1} + y_0^{n-1}y_1^{n}\bigg).\]
Multiplying the series $I$ and $N$ gives
\[ A_U(y_0, y_1) =\Exp\bigg(\frac{ (\LL +\LL^2)y_0y_1 -\LL^{1/2}(y_0+y_1)}{\LL -1}\sum_{n\geq 0} (y_0y_1)^n\bigg).\]

\subsection{Second proof: another dimensional reduction}

Recall that representations of the cut algebra $J_{W,I}$ are given by triples $(A_2,B_1,B_2)$, where 
$A_2:V_0\to V_1$, $B_1,B_2:V_1\to V_0$ are linear maps satisfying
\begin{equation}
B_1A_2B_2=B_2A_2B_1.
\label{eq:rel1}
\end{equation}
The pair $(A_2,B_2)$ gives a representation of the quiver 
$$C^2=(0,1;a:0\to1,b:1\to0).$$
Given the dimension vector $\al\in\cN^2$, let $R(J_{W,I},\al)$ be the space of representations of $J_{W,I}$ having dimension vector \al.
Let $R(C^2,\al)$ be the space of representations of $C^2$ having dimension 
vector \al. There is a forgetful map 
$$g:R(J_{W,I},\al)\to R(C^2,\al),\qquad (A_2,B_1,B_2)\to(A_2,B_2).$$
Its fibers are linear vector spaces. This map is equivariant with respect to the natural action 
of $\GG_\al=\GL_{\al_0}\xx\GL_{\al_1}$ on both sides.
Given a $C^2$-representation $M=(M_0,M_1;M_a,M_b)$, let $\rho(M)$ be the dimension of the fiber of $g$ over $M$. 
Let $M^0=(M_0;M_bM_a)$ and $M^1=(M_1;M_aM_b)$ be representations of the Jordan quiver $C^1$ (one vertex and one loop). 
Then it follows from \eqref{eq:rel1} that
$$\rho(M)=\dim\Hom_{C^1}(M^1,M^0).$$
More generally, for any two representations of $C^2$
$$M=(M_0,M_1;M_a,M_b),\quad N=(N_0,N_1;N_a,N_b)$$
we define
$$\rho(M,N)=\dim\Hom_{C^1}(M^1,N^0).$$

If $M$ is some representation of $C^2$ having dimension vector $\al$, then the contribution of its $\GG_\al$-orbit 
(\ie isomorphism class) to $[R(C^2,\al)]/[\GG_\al]$ is
$1/[\Aut M]$. The contribution of the preimage of its $\GG_\al$-orbit 
to $[R(J_{W,I},\al)]/[\GG_\al]$ is $\cL^{\rho(M)}/[\Aut M]$.

Let $M=\oplus M_i^{n_i}$ be a decomposition of a $C^2$-representation $M$ into the sum of indecomposable representations. Then by \cite[Theorem 1.1]{mozgovoy_motivicb}
$$[\Aut M]=[\End(M)]\cdot\prod_i(\cL\inv)_{n_i},$$
where $(q)_n=(q;q)_n=\prod_{k=1}^n(1-q^k)$ is the $q$-Pochhammer symbol.
Thus the contribution of the preimage of the $\GG_\al$-orbit of $M$ to $[R(J_{W,I},\al)]/[\GG_\al]$ is
\begin{equation}
\frac{\cL^{\rho(M)}}{[\Aut M]}
=\frac{\cL^{\rho(M,M)-h(M,M)}}{\prod_i(\cL\inv)_{n_i}}
=\frac{\prod_{i,j}\cL^{n_in_j(\rho(M_i,M_j)-h(M_i,M_j))}}{\prod_i(\cL\inv)_{n_i}}
,
\label{eq:4}
\end{equation}
where $h(M,N)=\dim\Hom(M,N)$ for any $C^2$-representations $M,N$.
We will compute the numbers $h(M,N)$ and $\rho(M,N)$ for indecomposable representations $M,N$ of $C^2$.
As is well known, the indecomposable representations of $C^2$ are the following:
\begin{enumerate}
	\item Representations $I_n$ of dimension $(n,n-1)$, $n\ge1$.
	\item Representations $P_n$ of dimension $(n-1,n)$, $n\ge1$.
	\item Representations $R_{t,n}=(I_n,J_{t,n})$, $n\ge1$, $t\in\cC$, of dimension $(n,n)$. There are also 
representations $R_{\infty,n}=(J_{0,n},I_n)$, $n\ge1$, of dimension $(n,n)$. Here $J_{t,n}$ denotes the Jordan block of size $n$ with value $t$ on the diagonal.
\end{enumerate}

\begin{rk}
Define duality on representations of $C^2$ by 
$$D(M_0,M_1;M_{12},M_{21})=(M_0\dual,M_1\dual;M_{21}\dual,M_{12}\dual).$$
Then
$\Hom(DM,DN)=\Hom(N,M)\dual$ and
$$D(I_n)=I_n,\quad D(P_n)=P_n,\quad D(R_{t,n})=R_{t\inv,n}.$$
Define equivalence (cyclic shift) by 
$$C(M_0,M_1;M_{12},M_{21})=(M_1,M_0;M_{21},M_{12}).$$
Then
$\Hom(CM,CN)=\Hom(M,N)$ and 
$$C(I_n)=P_n,\quad C(P_n)=I_n,\quad C(R_{t,n})=R_{t\inv,n}.$$
\end{rk}

The proofs of the following two propositions are easy exercises. 

\begin{prop}
We have
\begin{enumerate}
\item $h(R_{s,m},R_{t,n})=\min\set{m,n}$ if $s=t$ or $s,t\in\set{0,\infty}$. It is $0$ otherwise.
	\item $h(I_m,R_{t,n})=h(R_{t,n},P_m)=\case{
	\min\set{m-1,n}&t=0;\\
	\min\set{m,n}&t=\infty;\\
	0&t\in\cC^*.}$
	\item $h(R_{t,n},I_m)=h(P_m,R_{t,n})=\case{
	\min\set{m-1,n}&t=\infty;\\
	\min\set{m,n}&t=0;\\
	0&t\in\cC^*.}$
	\item $h(I_m,I_n)=h(P_m,P_n)=\min\set{m,n}$.
	\item $h(I_m,P_n)=h(P_n,I_m)=\min\set{m,n}-1$.
\end{enumerate}
\end{prop}

\begin{prop}
We have
\begin{enumerate}
	\item $\rho(R_{s,m},R_{t,n})=\min\set{m,n}$ if $s=t$ or $s,t\in\set{0,\infty}$. It is $0$ otherwise.
	\item $\rho(I_m,R_{t,n})=\rho(R_{t,n},P_m)=\case{
	\min\set{m-1,n}&t=0,\infty;\\
	0&t\in\cC^*.}$
	\item $\rho(R_{t,n},I_m)=\rho(P_m,R_{t,n})=\case{
	\min\set{m,n}&t=0,\infty;\\
	0&t\in\cC^*.}$
	\item $\rho(I_m,I_n)=\rho(P_n,P_m)=\min\set{m-1,n}$.
	\item $\rho(I_m,P_n)=\min\set{m,n}-1$.
	\item $\rho(P_n,I_m)=\min\set{m,n}$.
\end{enumerate}
\end{prop}

\begin{cor}
\label{cor:rho-h}
For any $C^2$-representations $M,N$,
let\[d(M,N)=\rho(M,N)-h(M,N).\] If $M,N$ are indecomposable, then
$$d(M,N)+d(N,M)
=\case{
1& M=I_m, N=P_n;\\
-1-\delta_{m,n}& M=I_m, N=I_n\text{ or }M=P_m, N=P_n;\\
0&\text{otherwise.}
}$$
\end{cor}

\begin{proof}[Proof of Theorem \ref{thm main universal result}]
We can decompose any $C^2$-representation as
$$M=I\oplus P\oplus\bigoplus_{t\in\cP^1}R_t=\bigoplus_{i\ge1} I_i^{m_i}\oplus\bigoplus_{i\ge1}P_i^{n_i}\oplus
\bigoplus_{t\in\cP^1}\bigoplus_{i\ge1}R_{t,i}^{r_i(t)}.$$
With the representation $M$ we associate partitions $\mu,\eta\in\lP$ and $\la(t)\in\lP$, for $t\in\cP^1$, in the following way:
$$\mu_k=\sum_{i\ge k}m_i,\qquad\eta_k=\sum_{i\ge k}n_i,\qquad\la_k(t)=\sum_{i\ge k}r_i(t).$$
Applying Corollary \ref{cor:rho-h}, we obtain
\begin{multline}
\label{eq:rho-h}
\rho(M,M)-h(M,M)=\sum_{i,j\ge1}m_in_j-\sum_{i\ge j\ge1}(m_im_j+n_in_j)\\
=-\oh\bigg(\Big(\sum_{i\ge1}(m_i-n_i)\Big)^2+\sum_{i\ge1}m_i^2+\sum_{i\ge1}n_i^2\bigg),
\end{multline}
an expression that we are going to denote by $d(\mu,\eta)$.
The dimension vectors of the summands of $M$ are given by
\begin{align*}
\udim I=&\bigg(\sum_{i\ge1} im_i,\sum_{i\ge1}(i-1)m_i\bigg)=(\n\mu,\n\mu-\mu_1),\\
\udim P=&\bigg(\sum_{i\ge1}(i-1)n_i,\sum_{i\ge1}in_i)\bigg)=(\n\eta-\eta_1,\n\eta),\\
\udim R_t=&\bigg(\sum_{i\ge1}ir_i(t),\sum_{i\ge1}ir_i(t)\bigg)=(\n{\la(t)},\n{\la(t)}).
\end{align*}

Applying equation \eqref{eq:4} we obtain
\begin{multline}
A_U=\sum_{\al\in\cN^2}(-\cL^\oh)^{(\al_0-\al_1)^2}\frac{[R(J_{W,I},\al)]}{[\GG_\al]}y^\al\\
=\sum_{\mu,\eta\in\lP}(-\cL^\oh)^{(\mu_1-\eta_1)^2}\frac{y_0^{\n\mu+\n\eta-\eta_1}y_1^{\n\mu+\n\eta-\mu_1}\cL^{d(\mu,\eta)}}
{\prod_{i\ge1}(\cL\inv)_{\mu_i-\mu_{i+1}}(\cL\inv)_{\eta_i-\eta_{i+1}}}
\sum_{\la:\cP^1\to\lP}\prod_{t\in\cP^1}f_{\la(t)},
\end{multline}
where
$$f_\la=\frac{(y_0 y_1)^{\n\la}}{\prod_{i\ge1}(\cL\inv)_{\la_i-\la_{i+1}}}.$$

By Hua formula's (see \cite{hua_counting} or \cite[Theorem 6]{mozgovoy_computational})
applied to the quiver with one loop, we obtain
$$f=\sum_{\la\in\lP}f_\la=\Exp\bigg(\frac{\cL}{\cL-1}\sum_{n\ge1}(y_0y_1)^n\bigg).$$
Therefore, using the geometric description of power structures (Section \ref{lambda}), we obtain
$$\sum_{\la:\cP^1\to\lP}\prod_{t\in\cP^1}f_{\la(t)}=
\Pow(f,[\cP^1])=\Exp\bigg(\frac{(\cL+1)\cL}{\cL-1}\sum_{n\ge1}(y_0y_1)^n\bigg).$$

On the other hand it follows from \eqref{eq:rho-h} that
$$(\mu_1-\eta_1)^2+2d(\mu,\eta)=-\sum_{i\ge1} m_i^2-\sum_{i\ge1} n_i^2$$
and therefore
\begin{equation}
A_U=\Exp\bigg(\frac{\cL+\cL^2}{\cL-1}\sum_{n\ge1}(y_0y_1)^n\bigg)
\sum_{\mu,\eta\in\lP}
\frac{y_0^{\n\mu+\n\eta-\eta_1}y_1^{\n\mu+\n\eta-\mu_1}(-\cL^\oh)^{-\sum_i(m_i^2+n_i^2)}}
{\prod_{i\ge1}(\cL\inv)_{m_i}(\cL\inv)_{n_i}},
\label{eq:A'}
\end{equation}
where we denote $m_i=\mu_i-\mu_{i+1}$, $n_i=\eta_i-\eta_{i+1}$.
Define
$$H(x,q^\oh)
=\sum_{n\ge0}\frac{(-q^\oh)^{-n^2}x^n}{(q\inv)_n}
=\sum_{n\ge1}\frac{(xq^\oh)^n}{(q)_n}
=\Exp\bigg(\frac{xq^\oh}{1-q}\bigg),$$
where the last equality follows from the Heine formula \cite{kac_quantum,mozgovoy_fermionic}.
Then the sum in \eqref{eq:A'} can be written in the form
\begin{multline*}
\sum_{(m_i)_{i\ge1},(n_i)_{i\ge1}}
\prod_{i\ge1}\frac{(y_0^iy_1^{i-1})^{m_i}(y_0^{i-1}y_1^i)^{n_i}(-\cL^\oh)^{-m_i^2-n_i^2}}
{(\cL\inv)_{m_i}(\cL\inv)_{n_i}}\\
=\prod_{i\ge1}\bigg(
\sum_{m\ge0}\frac{(y_0^iy_1^{i-1})^{m}(-\cL^\oh)^{-m^2}}{(\cL\inv)_{m}}
\sum_{n\ge0}\frac{(y_0^{i-1}y_1^{i})^{n}(-\cL^\oh)^{-n^2}}{(\cL\inv)_{n}}\bigg)\\
=\prod_{i\ge1}H(y_0^iy_1^{i-1},\cL^\oh)H(y_0^{i-1}y_1^{i},\cL^\oh)
=\Exp\bigg(\frac{\cL^\oh}{1-\cL}\sum_{i\ge1}(y_0^iy_1^{i-1}+y_0^{i-1}y_1^{i})\bigg).
\end{multline*}
The second proof of Theorem \ref{thm main universal result} is complete.
\end{proof}

\subsection{Decomposing the universal series}
\label{sec:decompose}

In this section, we decompose the product from Theorem \ref{thm main universal result}. We will say that a stability parameter $\ze$ is generic, if for any stable $J$-module $V$, we have $\ze\cdot\udim V\ne0$.
For generic stability parameter $\ze$, let $\gM_\ze^+(J,\al)$ (resp.\ $\gM_\ze^-(J,\al)$) denote the moduli stacks 
of $J$-modules $V$ such that $\udim V=\al$ and such that all the HN factors $F$ of $V$ with respect to the stability parameter \ze satisfy $\ze\cdot\udim F>0$ (resp. $<0$).
We put
$$A_\ze^\pm=\sum_{\al\in\cN_{Q_0}}[\gM_\ze^\pm(J,\al)]_\vir\cdot y^\al.$$ 

\begin{lem}
\label{lmm:decompose}
The generating series $A_\ze^\pm$ are given by 
\[A_\ze^\pm=\displaystyle\prod_{\substack{\alpha\in\Delta_+\\ \pm\ze\cdot\alpha<0}}A^\alpha,\]
where $A^\al=A^\al(y_0,y_1)$ were defined in Theorem \ref{thm main universal result}. We have
\[A_U=A_\ze^+A_\ze^-.
\]
\end{lem}
\begin{proof}
By Theorem \ref{thm_FP}, we have a factorization in $\That_Q$ (note that $\That_Q$ is commutative and we don't need to take the ordered product)
$$A_U=\prod_{\mu\in\cR}A_{\ze,\mu},$$
where $A_{\ze,\mu}$ were defined in Definition \ref{A_{Z,l}}.
Similarly we have $A_\ze^{\pm}=\prod_{\pm\mu>0}A_{\ze,\mu}$.

By Theorem \ref{thm main universal result}, we have $$A_{U}=\prod_{\al\in\De_+}A^\al,$$
where $A^\al$ contain only powers $y^{k\al}$, $k\ge0$. 
By the uniqueness of the factorizations from Lemma~\ref{factorlemma}, we obtain
$$A_{\ze,\mu}=\prod_{\substack{\al\in\De_+\\ \mu(\al)=\mu}}A^\al$$
and the statement of the lemma follows.
\end{proof}

\section{Motivic DT with framing}
\label{sec:framing}
\subsection{Framed quiver}
Let $Q$ be a quiver with a distinguished vertex $0\in Q_0$ and let $W$ be a potential.
We denote by $\Qtilde$ the corresponding framed quiver, the new quiver obtained from $Q$ by adding a new vertex $\infty$ and a single new arrow $\infty\to 0$. 
Let $\Jtilde=J_{\Qtilde,W}$ be the Jacobian algebra corresponding to the quiver with potential $(\Qtilde,W)$, where we view $W$ as a potential for $\Qtilde$ in the obvious way. Any $\wtl Q$-representation (resp.\ $\wtl J$-module) $\wtl V$ can be written as a triple $(V,\wtl V_\infty,s)$, where $V$ is a $Q$-representation (resp.\ $J$-module), $\wtl V_\infty$ is a vector space, and $s:\wtl V_\infty\to V_0$ is a linear map. We will always do this identification without mentioning.

The twisted motivic algebra $\That_Q$ of the original quiver sits as a subalgebra 
inside the algebra $\That_{\Qtilde}$ associated to the framed quiver $\Qtilde$. Note that in $\That_{\Qtilde}$ we have
\begin{equation}
y_\infty\cdot y^{(\al,0)}=(-\cL^\oh)^{-\al_0}\cdot y^{(\al,1)}=\cL^{-\al_0} \cdot y^{(\al,0)}\cdot y_\infty,
\label{eq:y infty rel}
\end{equation}
where we put
\[y_\infty=y^{(0,1)}.\]
In particular, $\That_{\Qtilde}$ is never commutative.

\subsection{Stability for framed representations}

Let $\ze\in\cR^{Q_0}$ be a vector, which we will refer to as the stability parameter.
\begin{defn} 
A $\Qtilde$-representation (resp.\ $\Jtilde$-module) $\wtl V$ with $\dim\wtl V_\infty= 1$ is said to be $\ze$-(semi)stable, if it is (semi)stable with respect to $(\ze,\ze_\infty)\in\cR^{\wtl Q}$ (see Definition \ref{semist}), where $\ze_\infty=-\ze\cdot\udim V$. Equivalently, the following conditions should be satisfied:
\begin{itemize}
\item
for any $\wtl Q$-subrepresentation (resp.\ $\wtl J$-submodule) $0\ne\wtl V'\sb\wtl V$ with $\wtl V'_\infty=0$, we have
$$\ze\cdot\udim V' \,(\le)\, 0;$$
\item
for any $\wtl Q$-quotient representation (resp.\ $\wtl J$-quotient module)  $\wtl V\twoheadrightarrow\wtl V''\neq 0$ with $\wtl V''_\infty=0$, we have
\[
\ze\cdot\udim V'' \,(\geq )\, 0.
\]
\end{itemize}
\end{defn}

As in Section \ref{sec:decompose}, a stability parameter $\ze\in\cR^{Q_0}$ is said to be {\it generic}, if for any stable $J$-module $V$ we have $\ze\cdot\udim V\ne0$.

\subsection{Motivic DT invariants with framing}
For a stability parameter $\ze\in\cR^{Q_0}$ and a dimension vector $\al\in \cN^{Q_0}$, let as before $\ze_\infty=-\ze\cdot\al$, $\wtl\al=(\al,1)$, and let
$$\gM_\ze(\wtl Q,\al)=[R_{(\ze,\ze_{\infty})}(\wtl Q,\wtl\al)/\GG_\al],\qquad
\gM_\ze(\wtl J,\al)=[R_{(\ze,\ze_\infty)}(\wtl J,\wtl\al)/\GG_\al]$$
denote the moduli stack of $\ze$-stable $\wtl Q$-representations (resp.\ $\wtl J$-modules) with dimension vector $\wtl\al$. The corresponding stacks for the trivial stability $\ze=0$ will be denoted by $\gM(\wtl Q,\al)$ and $\gM(\wtl J,\al)$.

\begin{rk}
Note that the stack $\gM_\ze(\wtl Q,\al)$ is slightly different from the stack
$\gM_{(\ze,\ze_\infty)}(\wtl Q,\wtl\al)$ which was defined in \eqref{eq:def moduli stacks} to be $[R_{(\ze,\ze_\infty)}(\wtl Q,\wtl\al)/\GG_{\wtl\al}]$. The same applies to the stacks of $\wtl J$-modules. 
\end{rk}


\begin{defn}
Let
$$\Atilde_U=
\sum_{\al\in\cN^{Q_0}}
\bigl[\gM\bigl(\Jtilde,\al\bigr)\bigr]_\vir \cdot y^{\wtl\al}\in \That_{\Qtilde},$$
where $\bigl[\gM\bigl(\Jtilde,\al\bigr)\bigr]_\vir$ is defined similarly to \eqref{eq:bbs1}. For any stability parameter $\ze\in\cR^{Q_0}$ let 
\[
\Atilde_\ze=
\sum_{\al\in \cN^{Q_0}}
\bigl[\gM_\ze\bigl(\Jtilde,\al\bigr)\bigr]_\vir \cdot y^{\wtl\al}\in \That_{\Qtilde},
\]
where $\bigl[\gM_\ze\bigl(\Jtilde,\wtl\al\bigr)\bigr]_\vir$ is defined similarly to \eqref{eq:bbs2}.
Let also, as in the Introduction, 
\[
Z_\ze=
\sum_{\al\in\cN^{Q_0}}
\bigl[{\MM}_\ze\bigl(\Jtilde,\al\bigr)\bigr]_\vir\cdot y^\al\in
\That_{Q}.
\]
\end{defn}

\subsection{Relating the universal and framed series}
In this subsection we assume that $Q$ is a symmetric quiver and therefore $\That_Q$ is commutative. 
The following theorem relates results of the previous section on the universal series to the framed invariants of this section. 

\begin{thm} 
\label{thm_framed_vs_nonframed}
For generic stability parameter \ze, we have
\begin{equation}
Z_\ze
=\frac{A_\ze^-(-\cL^{\oh}y_0,y_1,\dots)}{A_\ze^-(-\cL^{-\oh}y_0,y_1,\dots)},
\end{equation}
where $A_\ze^-$ were defined in Section \ref{sec:decompose}.
\end{thm}

This result is~\cite[Corollary 4.17]{mozgovoy_wall-crossing}.
In the rest of this subsection, we provide an alternative approach to this theorem. 
The main difference is that here we study just two stability parameters, while the result in \cite{mozgovoy_wall-crossing} was obtained by studying an infinite sequence of parameters between these two.

\begin{prop}\label{prop_filtration}
Let $\wtl V$ be a $\wtl{Q}$-representation \textup{(}resp.\ a $\wtl{J}$-module\textup{)} with $\dim\wtl V_\infty=1$. 
Then there exists the unique filtration 
\[
0=\wtl U^0\sb\wtl U^1\sb\wtl U^2 \sb\wtl U^3=\wtl V
\]
such that with $\wtl V^i=\wtl U^i/\wtl U^{i-1}$ we have
\begin{enumerate}
\item $\wtl V^{1}_\infty=0$ and 
all the HN factors $F$ of $V^1$ with respect to the stability parameter \ze satisfy $\ze\cdot\udim F>0$,
\vspace{3pt}
\item $\wtl V^2_\infty=1$ and $\wtl V^2$ is $\ze$-semistable,
\vspace{3pt}
\item $\wtl V^{3}_\infty=0$ and 
all the HN factors $F$ of $V^3$ with respect to the stability parameter \ze satisfy $\ze\cdot\udim F<0$.
\end{enumerate}
\end{prop}
\begin{proof}
We will work only with $\wtl Q$-representations.
We take sufficiently small $\eps >0$ and define the central charge
\[
Z_{\ze,\eps}(\wtl{\al})=-\ze\cdot\al+(\eps\n\al+\al_\infty)\sqrt{-1},\qquad \wtl\al=(\al,\al_\infty).
\]
Let $\wtl{W}$ be a $\wtl{Q}$-representation with $\dim \wtl{W}_\infty=1$.
For any submodule $\wtl{W}'={W}'$ of $\wtl{W}$ with $\wtl{W}'_\infty =0$, we have
\[
\ze\cdot \udim {W}' \gtrless 0 
\iff 
\arg Z_{\ze,\eps}(\udim\wtl{W}')
\gtrless
\arg Z_{\ze,\eps}(\udim\wtl{W}).
\]
Hence $\wtl{W}$ is $Z_{\ze,\eps}$-stable if and only if it is $\ze$-stable.
Then, the Harder-Narashimhan filtration for $Z_{\ze,\eps}$-stability is the required filtration.
\end{proof}

The filtration from Proposition \ref{prop_filtration} 
induces the following factorization in the same way as Theorem \ref{thm_FP}:
\begin{prop}\label{prop_factorization1}
We have
$$\wtl A_U =  A_\ze^+ \cdot \wtl A_\ze \cdot A_\ze^-$$
in the motivic algebra $\That_{\wtl Q}$.
\end{prop}

\begin{prop}
\label{wtl A}
\[
\wtl A_U =  A_U\cdot y_\infty.
\]
\end{prop}
\begin{proof}
Any $\wtl{Q}$-module (resp.\ $\wtl J$-module) $\wtl V$ with $\dim\wtl V_\infty=1$ and $\udim V=\al$ has a unique filtration 
$$0\sb V\sb\wtl V$$
with 
$$\wtl V/V\simeq S_\infty,$$
the simple module concentrated at the vertex $\infty$.
Thus the factorization follows.
\end{proof}

\begin{proof}[Proof of Theorem \ref{thm_framed_vs_nonframed}]
We have
\begin{align} 
\wtl A_\ze
&=(A_\ze^+)^{-1}\cdot\wtl A_U\cdot(A_\ze^-)^{-1} 
&(\text{Proposition \ref{prop_factorization1}})\notag\\
&=(A_\ze^+)^{-1}\cdot(A_\ze^+\cdot A_\ze^-\cdot y_\infty)\cdot(A_\ze^-)^{-1}
&(\text{Prop.~\ref{wtl A} and Lemma \ref{lmm:decompose}})\notag\\
&=y_\infty \cdot \frac{A_\ze^-(\cL y_0,y_1,\dots)}{A_\ze^-(y_0,y_1,\dots)}.
& (\text{Equation \eqref{eq:y infty rel}})
\label{eq_2}
\end{align}
It follows from \eqref{eq:y infty rel} that $y_\infty\cdot Z_\ze(-\cL^\oh y_0,\dots)=\wtl A_\ze$.
Combining this with \eqref{eq_2} we get the statement of Theorem~\ref{thm_framed_vs_nonframed}.
\end{proof}

\subsection{Application to the conifold}
The following theorem is the main result of this section, announced as Theorem~\ref{announce_thm_main_res}. Let $(Q,W)$ be the conifold quiver with potential.

\begin{thm} For generic $\ze\in\R^2$, 
\begin{equation}
Z_\ze(y_0,y_1)= \prod_{\substack{\al\in\Delta_+ \\ \ze\cdot\al<0}} Z_{\al}(y_0, y_1),
\label{eq main res form}
\end{equation}
with 
\[
Z_\al(-y_0,y_1)\!=\!\begin{cases}
\displaystyle\prod_{j=0}^{\al_0-1}
\left(1-\cL^{-\frac{\al_0}2+\oh+j}y^\al\right)&\al\in\De^\re_+\\
\displaystyle\prod_{j=0}^{\al_0-1}
\left(1-\cL^{-\frac{\al_0}2+1+j}y^\al\right)\inv\left(1-\cL^{-\frac{\al_0}2+2+j}y^\al\right)\inv&\al\in\De^\im_+
\end{cases}
\]
\label{thm_main_res}
\end{thm}
\begin{proof} Substituting the result of Lemma \ref{lmm:decompose} into Theorem~\ref{thm_framed_vs_nonframed},
we get the product form~\eqref{eq main res form}, with \[Z_\al(y_0,y_1)=A^\al(-\cL^{\oh}y_0,y_1)/A^\al(-\cL^{-\oh}y_0,y_1).\]
Now use the expression for $A^\al$ from Theorem~\ref{thm main universal result}.
\end{proof}

\section{DT/PT series}
\subsection{Chambers and the moduli spaces for the conifold}

Let $(Q,W)$ be the conifold quiver with potential.
In the space $\R^2$ of stability parameters, consider the lines
\begin{align*}
L_+(m)&=\{(\ze_0,\ze_1)\mid m\,\ze_0+(m-1)\ze_1=0\}\quad (m\geq 1),\\
L_\infty&= \{(\ze_0,\ze_1)\mid \ze_0+\ze_1=0\},\\
L_-(m)&= \{(\ze_0,\ze_1)\mid m\,\ze_0+(m+1)\ze_1=0\}\quad (m\geq 0).
\end{align*}
It is immediately seen that these are exactly the lines orthogonal to the roots in $\Delta_+$ 
with respect to the standard inner product. 
Let $L\subset\R^2$ denote the union of this countable set of lines. The complement of $L$ 
in $\R^2$ is a countable union of open cones.
Denote by $Y^+$ the flop of $Y$ along the embedded rational curve. 

\begin{thm} {\rm\cite[Lemma 3.1 and Propositions 2.10-2.13]{nagao_counting}} \label{thm_desrc}
The set of generic parameters in $\cR^2$ is the complement of the union $L$ of the lines defined above. 
\begin{enumerate}\renewcommand{\theenumi}{\roman{enumi}}
\item For $\ze$ with $\ze_0<0$ and $\ze_1<0$, the moduli spaces $\MM_\ze(\Jtilde,\al)$ are the NCDT moduli spaces, the moduli spaces of cyclic $J$-modules from~\cite{szendroi_non-commutative}.
\item For $\ze$ near the line $L_\infty$ with $\ze_0<\ze_1$ and $\ze_0+\ze_1<0$, the moduli spaces $\MM_\ze(\Jtilde,\al)$ are the commutative DT moduli spaces of $Y$ from~\cite{maulik_gromov-witten}, 
the moduli spaces of subschemes on $Y$ with support in dimension at most 1.
\item For $\ze$ near the line $L_\infty$ with $\ze_0<\ze_1$ and $\ze_0+\ze_1>0$, the moduli 
spaces $\MM_\ze(\Jtilde, \al)$ are the PT moduli spaces of $Y$ introduced in~\cite{pandharipande_curve}; 
these are moduli spaces of stable rank-1 coherent systems. 
\item For $\ze$ near the line $L_\infty$ with $\ze_0>\ze_1$ and $\ze_0+\ze_1<0$, the moduli spaces $\MM_\ze(\Jtilde, \al)$ are the commutative DT moduli spaces of the flop $Y^+$. 
\item For $\ze$ near the line $L_\infty$ with $\ze_0<\ze_1$ and $\ze_0+\ze_1>0$, the moduli spaces $\MM_\ze(\Jtilde, \al)$ are the PT moduli spaces of the flop $Y^+$. 
\item For $\ze$ with $\ze_0>0$ and $\ze_1>0$, the moduli space $\MM_\ze(\Jtilde, \al)$ consists of a point for $\al=0$ and is otherwise empty.
\end{enumerate}
\end{thm}

\begin{rk} Note that ``near'' in the above statements means sufficiently near depending on the 
dimension vector $(\al,1)$. 
\end{rk} 

\subsection{Motivic PT and DT invariants}

\begin{prop}
The refined partition functions of the resolved conifold $Y$ for the DT and PT chambers are given by
\begin{equation}
Z_{\rm PT}(-y_0, y_1) = \prod_{m\geq 1}\prod_{j=0}^{m-1} \left(1-\cL^{-\frac m2+\oh+j}y_0^m y_1^{m-1}\right)
\label{PTseries}
\end{equation}
and
\begin{multline}
Z_{\rm DT}(-y_0, y_1) = Z_{\rm PT}(-y_0, y_1)\\
\cdot\prod_{m\geq 1}\prod_{j=0}^{m-1}
\left(1-\cL^{-\frac m2+1+j}y_0^m y_1^m\right)^{-1} 
\left(1-\cL^{-\frac m2+2+j}y_0^m y_1^m\right)^{-1}.
\label{DTseries}
\end{multline}
\end{prop}
\begin{proof}
Let $\ze=(-1+\eps,1)$, $0<\eps\ll1$, be some stability corresponding to PT moduli spaces. Then
$$\sets{\al\in\De_+}{\ze\cdot\al<0}=\sets{(m,m-1)}{m\ge1}.$$
Applying Theorem~\ref{thm_main_res} we obtain
$$Z_{PT}(-y_0,y_1)=\prod_{\ze\cdot\al<0}Z_\al(-y_0,y_1)
=\prod_{m\geq 1}\prod_{j=0}^{m-1} \left(1-\cL^{-\frac m2+\oh+j}y_0^m y_1^{m-1}\right).$$
The proof of the second formula is similar.
\end{proof}

Let us re-write these formulae in the perhaps more familiar large radius parameters $T=y_1\inv$, $s=y_0y_1$, 
corresponding to the cohomology class of a point and a curve on the geometry $Y$. We obtain
\begin{equation}
Z_{\rm PT}(-s, T) = \prod_{m\geq 1}\prod_{j=0}^{m-1} \left(1- \cL^{-\frac m2+\oh+j} s^m T \right)
\label{PTseries_largerad}
\end{equation}
and
\begin{equation}
Z_{\rm DT}(-s, T) = Z_{\rm PT}(-s, T)\cdot\prod_{m\geq 1}\prod_{j=0}^{m-1}
\left(1-\cL^{-\frac m2+1+j}s^m\right)^{-1} 
\left(1-\cL^{-\frac m2+2+j}s^m\right)^{-1}.
\label{DTseries_largerad}
\end{equation}
The specializations at $\LL^{\oh}=1$ are the PT and DT series of the resolved conifold respectively, given
by the standard expressions
\[ \ub Z_{\rm PT}(-s, T) = \prod_{m\geq 1} \left(1-Ts^m\right)^m=\Exp\left(\frac{-T}{(s^\oh-s^{-\oh})^2}\right)
\]
and
\[\ub Z_{\rm DT}(-s, T) =  M(s)^2 \prod_{m\geq 1} \left(1-Ts^m\right)^m,\]
with $M(s)=\prod_{m\geq 1}(1-s^m)^{-m}$ the MacMahon function, and $\ub{\phantom 1}$ denoting generating series 
of numerical (as opposed to motivic) invariants.

Wall crossing at the special wall~$L_\infty$ is the PT/DT wall crossing of~\cite{pandharipande_curve}. 
On the PT side, the coefficient of the $T^0$ term is just $1$, since if there is no curve present, the only 
possible PT pair consists of the structure sheaf of $Y$ (with zero map). On the DT side,
the moduli space with zero curve class is the moduli space of ideal sheaves of point clusters on $Y$, in other
words the Hilbert scheme of points of $Y$. Hence the ratio of the $T^0$ terms gives the generating function of 
virtual motives of the Hilbert scheme of points of $Y$~\cite{behrend_motivic}:
\[\sum_{n=0}^\infty [Y^{[n]}]_\vir (-s)^n = \prod_{m\geq 1} \prod_{j=0}^{m-1} \left(1-\LL^{1+j-\frac{m}{2}}s^m\right)^{-1}\left(1-\LL^{2+j-\frac{m}{2}}s^m\right)^{-1}.\]
At $\LL^{\oh}=1$, we obtain the MNOP result $M(s)^2$. Note in particular that, as proved in~\cite{behrend_motivic}
but contrary to the speculations of~\cite{dimofte_refined}, the motivic refinement is not a square, 
though both products are combinatorial refinements of the usual MacMahon series.

\begin{rk} Note that our results in fact imply a full factorization
\begin{equation}\label{eq:refinedPTDT}
 Z_{\rm DT}(s, T) = \left(\sum_{n=0}^\infty [Y^{[n]}]_\vir s^n \right)Z_{\rm PT}(s, T), 
\end{equation}
with the middle sum being a product of refined MacMahon series as in~\cite{behrend_motivic}. This is a motivic
analogue of the factorization
\begin{equation}\label{eq:PTDT}  \ub Z_{X, {\rm DT}}(s, T) = M(s)^{e(X)}\ub Z_{X, {\rm PT}}(s, T)
\end{equation}
conjectured for a quasi-projective Calabi--Yau threefold X in~\cite{maulik_gromov-witten, pandharipande_curve}, 
proved in~\cite{bridgeland_hall} following earlier proofs of a version of this statement in~\cite{stoppa_hilbert, toda_curve}.
In general, the only definition we have of the motivic $Z_{X, {\rm DT}}$ and $Z_{X, {\rm PT}}$ is through the 
partially conjectural setup of~\cite{kontsevich_stability}. Assuming that relevant parts of~\cite{kontsevich_stability} 
are put on a firm footing, including the integration ring homomorphism from the motivic
Hall algebra to the motivic quantum torus, it seems likely that the proof of~\cite{bridgeland_hall} can be adapted to prove the motivic 
version~\eqref{eq:refinedPTDT} in general.
\end{rk}

\subsection{Connection with the refined topological vertex}\label{sec:ref-top-vx}
The standard way to compute the unrefined PT series of the resolved conifold $Y$ is via
the topological vertex~\cite{aganagic_topological}. From the toric combinatorics, we obtain the formula
\[ \ub Z^{\rm vertex}_{\rm PT}(-s, T) = \sum_{\lambda} \ub C_{\lambda\emptyset\emptyset}(s) \ub C_{\lambda^t\emptyset\emptyset}(s) (-T)^{|\lambda|},
\]
see e.g.~\cite[(63)]{iqbal_refined}.
On the right hand side, the sum runs over all partitions; for a partition $\lambda$, 
$\lambda^t$ denotes the conjugate partition, and $\ub C_{\lambda\mu\nu}(s)$ is the
topological vertex expression of~\cite{aganagic_topological}. In the case when $\mu=\nu=\emptyset$, 
$\ub C_{\lambda\emptyset\emptyset}(s)$ can be expressed as a simple Schur function, and then Cauchy's identity
immediately gives
\[ \ub Z^{\rm vertex}_{\rm PT}(-s, T) = \prod_{m\geq 1} \left(1-Ts^m\right)^m = \ub Z_{\rm PT}(-s, T).\]
In mathematical terms~\cite{maulik_gromov-witten}, this equality (or rather its DT version) 
expresses torus localization, the combined expression $M(s)\ub C_{\lambda\mu\nu}(s)$ being the generating function 
of 3-dimensional partitions with given 2-dimensional asymptotics along the coordinate axes. 

The refined PT partition function as computed by the refined topological vertex is~\cite[(67)]{iqbal_refined}
\[
Z^{\rm vertex}_{\rm PT}(t,q,T) =  \sum_{\lambda} C_{\lambda\emptyset\emptyset}(q,t) C_{\lambda^t\emptyset\emptyset}(t,q) (-T)^{|\lambda|},
\]
where $C_{\lambda\mu\nu}(q,t)$ is now the refined topological vertex expression.
Using the Cauchy identity again, this sum reduces to~\cite[(67)]{iqbal_refined}
\begin{align}
Z^{\rm vertex}_{\rm PT}(t,q,T)
&=\prod_{i,j\geq 1}\left(1-T q^{i-\frac{1}{2}}t^{j-\frac{1}{2}}\right) \label{vertex_ref_PT} \\
&=\exp\bigg(\sum_{n\ge1}\frac{-T^n}{n(q^{\frac n2}-q^{-\frac n2})(t^{\frac n2}-t^{-\frac n2})}\bigg)\nonumber\\
&=\Exp\left(\frac{-T}{(q^\oh-q^{-\oh})(t^\oh-t^{-\oh})}\right)\nonumber.
\end{align}

\begin{prop}
We have
$$Z^{\rm vertex}_{\rm PT}(t,q,T)=Z_{PT}(-s,T),$$
when we make the change of variables $q=\cL^\oh s$, $t=\cL^{-\oh} s$, with $(qt)^\oh=s$.
\end{prop}
\begin{proof} This is immediate when we compare~\eqref{PTseries_largerad} with~\eqref{vertex_ref_PT}.
\end{proof}

Thus we obtain a proof of the ``motivic=refined'' correspondence~\cite{dimofte_refined} in this example.
Note however that the situation is very different from the unrefined story: here we are not giving
a full mathematical interpretation of the refined topological vertex expression
$C_{\lambda\mu\nu}(q,t)$; indeed, it remains a very interesting problem to find one. We are only checking 
that the results agree in the particular case of the resolved conifold.

\subsection{Connection to the cohomological Hall algebra}
An alternative to considering the refined motivic invariants is to consider the mixed Hodge modules of 
vanishing cycles~\cite{dimca_milnor, kontsevich_cohomological}. The description as the vanishing locus 
of the trace of the potential endows the moduli spaces 
${\MM}_\ze\bigl(\Jtilde,\wtl\al\bigr)$ with the mixed Hodge modules
of vanishing cycles of the trace function. Recently, the cohomologies of the moduli spaces with
coefficients in these mixed Hodge modules have been organized into an algebra in~\cite{kontsevich_cohomological}, 
the (critical) cohomological Hall algebra. Replacing $\LL$ by $q$ in all our formulae, we obtain generating 
series of E-polynomials of these mixed Hodge modules, the analogues of the formulae of~\cite{dimca_milnor} in our 
situation.

\section*{Acknowledgements}
The authors thank Tom Bridgeland, Jim Bryan, and Tamas Hausel for helpful discussions. 
K.N.\ and B.Sz.\ wish to thank the Isaac Newton Institute, Cambridge for its hospitality, where
parts of this paper were written.
K.N.\ is supported by the Grant-in-Aid for Research Activity Start-up (No.\ 22840023) and for 
Scientific Research (S) (No.\ 22224001). S.M.\ is supported by EPSRC grant EP/G027110/1.

\providecommand{\bysame}{\leavevmode\hbox to3em{\hrulefill}\thinspace}
\providecommand{\href}[2]{#2}

\enlargethispage{5\baselineskip}
 
\end{document}